\documentclass[11pt]{article}
\usepackage{amsfonts} 

\usepackage{amsmath, amsfonts, amssymb}
\usepackage{mathrsfs}
\usepackage{theorem}
\usepackage{bm}

\newtheorem{thm}{Theorem}[section]

\newtheorem{lem}[thm]{Lemma}
\newtheorem{cor}[thm]{Corollary}

\usepackage{color}

{\theorembodyfont{\rmfamily} \newtheorem{rem}[thm]{Remark}}
{\theorembodyfont{\rmfamily} \newtheorem{exam}{Example}[section]}

\newcommand{\ra}{\rightarrow}

\def\R{\mathbb R}

\def\F{\mathscr F}
\def\d{\text{\rm{d}}}
\def\E{\mathbb E}
\def\p{\mathbb P}\def\e{\text{\rm{e}}}

\def\La{\Lambda}

\def\S{\mathcal M}

\def\al{\alpha}

\newcommand{\fin}{\hspace*{\fill}\rule{0.3em}{1ex}}
\newenvironment{proof}{{\bf \noindent Proof.}}{\fin}

\numberwithin{equation}{section}

\allowdisplaybreaks

\begin{document}

\title{Asymptotic Behavior of SIRS Models in State-dependent Random Environments 
}

\author{
{   Jianhai Bao$^{b)}$,  Jinghai Shao$^{a)}$}\\
\footnotesize{$^{a)}$ Center for Applied Mathematics, Tianjin
University, Tianjin 300072, China}\\
\footnotesize{$^{b)}$Department of Mathematics, Swansea University,
Singleton Park, SA2 8PP, UK}\\
 \footnotesize{jianhaibao@csu.edu.cn,
shaojh@tju.edu.cn}} \maketitle

\begin{abstract}
The extinction and persistence of infective individuals are closely related to the random change of the environment. In this paper, via the random/stochastic SIRS models, we analyze qualitatively and quantitatively  the impact caused by the random change of the environment. Our contributions consist in (i) giving some sufficient  conditions on
extinction (persistence) of the infectious individuals even though
they are persistent (resp. extinct) in certain fixed environments;
(ii) revealing the influence of  random switching of incidence
functions on extinction for the infectious individuals,
which has not been studied before; (iii)
establishing a criterion to judge extinction of the infectious
individuals for a range of random/stochastic SIRS models with
state-dependent switching via a stochastic comparison  for
functionals of jump processes. Moreover, some
 examples are set to illustrate the applications of our theory.
\end{abstract}



\section{Introduction}

Let $S_t,I_t,R_t$ be the number of susceptible individuals,
infective individuals, and  removed individuals at time $t$, and
$N_t=S_t+I_t+R_t$ be the totality of the population. Assume that
infectious disease can cause additional mortality, and that an
infectious individual can recover with a loss of immunity. Since the
pioneer work due to Kermack-McKendrick \cite{KM}, the SIRS model has
been extensively investigated on, e.g., stability, extinction,
persistence, Hopf bifurcation, to name a few. Different diseases
have been discovered to be described via different incidence
functions. So numerous variants of incidence rate functions have
been put froward to fit well in the practical situation;  see, e.g.,
\cite{CS,DD,HD,KM,LLI,Tang}. In order to incorporate the effect of
behavioral changes and prevent unbounded contact rates, \cite{LLI}
consider an SIRS model with a nonlinear incidence rate function in
the form
\begin{equation}\label{general}
\left\{ \begin{split}
  \frac{\d S_t}{\d t}&=\Lambda -\mu S_t- S_tG(I_t)+\gamma R_t,\\
  \frac{\d I_t}{\d t}&=S_tG(I_t)-(\mu+\nu+\delta)I_t,\\
  \frac{\d R_t}{\d t}&=\nu I_t-(\mu+\gamma)R_t.
\end{split}\right.
\end{equation}
The precise interpretations on the parameters in \eqref{general} are
presented as follows: $\Lambda>0$ means the recruitment rate of the
population; $\mu>0$ denotes the natural death rate of the
population; $\delta$ stands for the disease inducing death rate;
$\gamma>0$ signifies the rate at which recovered individuals lose
immunity and return to the susceptible class; $\nu>0$ stipulates the
natural recovery rate of the infectious individuals; $SG(\cdot)$
($G:\R\rightarrow\R_+$) manifests the incidence rate per infective
individual. In particular, \cite{LLI} initiated a nonlinear
incidence function in the form
\begin{equation}\label{3.1}
G(x)=\frac{\beta x^\ell }{1+a x^h},~~~~x>0,
\end{equation}
where $\beta,\ell, h>0$ and $a\ge0,$ $\beta x^\ell$ measures the
infection force of the disease and $1/(1+ax^h)$ represents the
inhibition effect from the behavioral change of the susceptible
individuals when the number of infectious individuals increases. In
\eqref{3.1}, by taking $\ell=1$ and $a=0,$ \eqref{general} goes back
to an SIRS model with bilinear incidence rates (see e.g.
\cite{He,KM}).  \eqref{general} is said  to be the SIRS model with
unbounded incidence function for $\ell>h$, saturated incidence
function for $\ell=h$, and nonmonotone incidence function
 for $\ell<h$, respectively; see  e.g.  \cite{HD,LLI,NK,Tang,XR} and
references within.

The deterministic SIRS models \eqref{general} have been extended  in
several different ways into stochastic or random counterparts. One
of them is to perturb the deterministic models by white noises, see,
for instance, \cite{Cai15,LO,NK,TBV,ZJ}  upon  asymptotic analysis.
Whereas,
 with regard to  deterministic SIRS models or
  stochastic counterparts perturbed by white noises,  the environment   is
assumed to be constant. As we know, the evolution of the diseases
may heavily depend on the environment conditions such as
temperature, humidity, etc. 
 So, in practical situations, it is prerequisite
to take   the random changes of environmental conditions and their
effects upon the spread of the disease into account, where one of
natural and important questions is   to justify the persistence or
extinction of the disease. So, another extension of deterministic
SIRS models is to perturb via the telegraph noises, which is, in
general, called SIRS models with Markov switching or in random
environments; see  e.g.
 \cite{GGMP,GLM,HD}. For
population dynamical systems   in random environments,   we refer to
e.g.  \cite{BS16,DDY,DDD,LY}.

In the present work we are interested in  three kinds of SIRS models
below.

 \noindent {{\bf Model I:}} Taking impacts of the random environments into
consideration, we first consider the following state-independent
regime-switching SIRS model:
\begin{equation}\label{1.1}
\left\{\begin{split}  \frac{\d}{\d t} S_t &=
\La_{\alpha_t}-\mu_{\al_t}S_t-
  G(I_t,\al_t)S_t+\gamma_{\al_t}R_t \\
 \frac{\d }{\d t}I_t &= G(I_t,\al_t)S_t-(\mu_{\al_t}+\nu_{\al_t}+\delta_{\al_t})I_t \\
 \frac{\d }{\d t}R_t &= \nu_{\al_t}
 I_t-(\mu_{\al_t}+\gamma_{\al_t})R_t
\end{split}\right.
\end{equation}
with the initial datum
$(S_0,I_0,R_0)=(s_0,i_0,r_0)\in\R_+^3:=\{(x,y,z)\in\R^3:
x>0,y>0,z>0\}$ and $\al_0=a_0\in\S=\{1,2,\ldots,M\}$ for some
integer $M<\infty$. Herein,    $(\alpha_t)_{t\ge0}$ is a
continuous-time Markov chain with the state space $\S$ and the
transition probability specified by
\begin{equation}\label{v4}
\mathbb{P}(\al_{t+\triangle}=j|\al_t=i)=
\begin{cases}
q_{ij}\triangle+o(\triangle),~~~~~~~~i\neq j\\
1+q_{ii}\triangle+o(\triangle),~~~i=j
\end{cases}
\end{equation}
provided $\triangle\downarrow0$   and  inducing the $Q$-matrix
$Q=(q_{ij})_{i,j\in\S}$; $G:\R\times\S\rightarrow\R_+$ is continuous
w.r.t. the first variable; $\mu_i,\, \,\gamma_i,\,\delta_i,\nu_i$,
$i\in\S$, are positive constants, whose precise implications are
explicated as in \eqref{general}.

\noindent{{\bf Model II:}} We are still interested in \eqref{1.1},
whereas $(\al_t)_{t\ge0}$ is a jump process with the state space
$\S$ and the transition kernel stipulated as, for any $i,j\in\S$ and
$ x\in\R_+^3$,
\begin{equation}\label{jump}
\p(\al_{t+\Delta}=j|\al_{t}=i,X_t=x)=\begin{cases}
  q_{ij}(x)\Delta+o(\Delta), &i\neq j\\
  1+q_{ii}(x)\Delta+o(\Delta),&i=j
\end{cases}
\end{equation}
whenever $\Delta\downarrow0$, where $X_t=(S_t,I_t,R_t)\in \R_+^3$.

\noindent{\bf Model III:} Taking the influences  of   the
state-dependent random environments and  stochastic perturbations
into account, we focus on the
  following SIRS model
\begin{equation}\label{se3} \left\{\begin{split}
  \d S_t &=\Big(\La_{\alpha_t}-\mu_{\al_t}S_t-
 \frac{\beta_{\al_t}I_t}{f(I_t)}S_t+\gamma_{\al_t}R_t\Big)\d t- \mu_{\al_t}^eS_t\d B_t^{(1)}- \frac{\beta_{\al_t}^eI_t}{f(I_t)}S_t\d B_t^{(2)}\\
  \d I_t&=\Big(\frac{\beta_{\al_t}^eS_t}{f(I_t)}-(\mu_{\al_t}+\nu_{\al_t}+\delta_{\al_t})\Big)I_t\d t-\mu_{\al_t}^eI_t\d B_t^{(1)}+
  \frac{\beta_{\al_t}^eS_t}{f(I_t)}I_t\d B_t^{(2)}\\
 \d R_t &=(\nu_{\al_t}
  I_t-(\mu_{\al_t}+\gamma_{\al_t})R_t)\d t-\mu_{\al_t}^eR_t\d B_t^{(1)}
\end{split}\right.
\end{equation}
with the initial datum $(S_0,I_0,R_0)=(s_0,i_0,r_0)\in\R_+^3$ and
$\al_0=a_0\in\S$. Herein, $ \mu_i^e,\beta_i^e\ge0;$
$f:\R_+\rightarrow\R_+$ satisfies ({\bf A3}) below;
$B_t=(B_t^{(1)},B_t^{(2)})$ is a $2$-dimensional Brownian motion
defined on probability space
$(\Omega,\F,(\F_t)_{t\ge0},\mathbb{P})$, $(\al_t)_{t\ge0}$ is a
continuous time jump process determined by \eqref{jump}, and the
other quantities are defined exactly as in \eqref{1.1}.

Based on the three models above, in this work we aim  to
\begin{enumerate}
\item[({\bf i})] provide some sufficient  conditions to guarantee
the
extinction (persistence) of the infectious individuals even though
they are persistent (resp. extinct) in certain fixed environments;
\item[({\bf ii})]
illustrate  the impacts of  {\it random switching of incidence
functions} on extinction for the infectious individuals;
\item[({\bf iii})]
  establish  a criterion to judge extinction of the infectious
individuals for random/stochastic SIRS models with {\it
state-dependent regime switching.}
\end{enumerate}

Now we make the following remarks:
\begin{enumerate}
\item[({\bf1})]
The SIRS model \eqref{1.1} and \eqref{v4} enjoys the following
features: (i) It owns the incidence functions of one kind (e.g.
$G(x,i)= \beta_ix/f(x)$), which however takes different values in
different environments; (ii) It allows the incidence functions
(e.g., $G(x,i)=\frac{\beta_i x^i}{1+ax^2}$) to be distinctive in
different environments. See  Theorems \ref{t3.1} and \ref{t2.3},
corollaries \ref{co1}, \ref{co}, \ref{co2}, and Examples
\ref{exam1}, \ref{exam2}, and \ref{ex2} for more details.

\item[({\bf2})] Compared with the SIRS model \eqref{1.1}
and \eqref{v4}, there are essential challenges to cope with the
model \eqref{1.1} and \eqref{jump}. For this setup, one of the
challenges is that the classical ergodic theorem for continuous-time
Markov chains does not work any more due to the fact that
$(\al_t)_{t\ge0}$ is merely a jump process rather than a Markov
process. To get over such difficulty, we adopt a stochastic
comparison approach (see Lemma \ref{lm} for further details) for
functionals of the jump process $(\alpha_t)_{t\ge0}$. Moreover, we
  provide  explicit criteria
on the extinction/persistence of the infectious individuals; see
Theorems \ref{state} and \ref{th} and Examples \ref{exam3} and
\ref{exam4}.

\item[({\bf3})]  Since the
totality $N_t=S_t+I_t+R_t$ is variable, the approaches adopted to
cope with \eqref{1.1} and \eqref{v4} (or \eqref{jump}) is
unavailable for the model \eqref{se3} and \eqref{jump}. So some
tricks need to be put forward to investigate extinction of the
infectious individuals; see Theorem \ref{sto} for further details.
\end{enumerate}

The content of this paper is arranged as follows: Section \ref{sec2}
is concerned with impacts of state-independent random environments
on
    existence and
persistence for the infectious individuals solved by \eqref{1.1} and
\eqref{v4}; Section \ref{sec4} focuses on the influence of
state-dependent random environments  upon extinction/persistence of
the infectious individuals determined by \eqref{1.1} and
\eqref{jump}; Section \ref{sec3} is devoted to extending the random
SIRS model \eqref{1.1} and \eqref{v4} (or \eqref{jump}) into the
stochastic counterpart \eqref{se3} and \eqref{jump} and
 providing some sufficient conditions to justify
extinction of the infectious individuals.

\section{Impacts of state-independent  random environments}\label{sec2}
In the SIRS model \eqref{1.1} and \eqref{v4}, the transition rates
of the continuous time Markov chain $(\al_t)_{t\ge0}$  is
state-independent. For related analysis of
stochastic systems with state-independent random environments,  we
refer to e.g. \cite{CS,DDY,DDD,GGMP,HD,LLC,LY,MY,Tang} and
references therein. 

Let $\check{\La}=\max_{i\in\S} \La_i$ and $\hat\La=\min_{i\in\S}
\La_i$.   The other quantities $\check{\mu}$, $\hat\mu$,
$\check{\beta}$, $\hat\beta$, $\cdots$, are defined analogously.
Assume that
\begin{enumerate}
\item[({\bf A1})] For each $i\in\S$, $G(\cdot,i):\R\rightarrow\R_+$
is locally Lipschtz continuous and that  there exists a constant
$c>0$ such that $ G(x,i)\le c\,(1+|x|),x\in\R;$

\item[({\bf A2})] The continuous-time Markov chain $(\al_t)$
 is irreducible and positive
recurrent with the invariant probability measure
$\pi=(\pi_1,\cdots,\pi_M).$

\end{enumerate}

\begin{rem}\label{re1}
It is easy to check that the linear incidence rate (i.e.,
$G(x,i)=\beta_i x$), the saturated incidence rate (i.e.,
$G(x,i)=\frac{\beta_i x^\ell}{1+ax^\ell}, \ell>0$), the nonmonotone
incidence rate (i.e., $G(x,i)=\frac{\beta_i x^\ell}{1+ax^h}, 0\le
\ell<h)$, 
and the ``media
coverage'' incidence rate (e.g., $G(x,i)=\beta_ix\e^{-\alpha x},
\alpha>0$) fulfill the assumption ({\bf A1}) above.
\end{rem}

The lemma below shows that the unique solution to \eqref{1.1} and
\eqref{v4}  lies in the positive quadrant and implies that the
totality of the population (i.e., $N_t$) has an upper bound.

\begin{lem}\label{t2.1}
Assume that $({\bf A1})$ holds. Then, \eqref{1.1} and \eqref{v4}
  has a unique strong solution
$(S_t,I_t,R_t)\in\R_+^3$  with   the initial value
$(s_0,i_0,r_0)\in\R^3_+$. Moreover,
\begin{equation}\label{2.2}
    N_t\leq N_0\e^{-\int_0^t\mu_{\al_s}\d s}+\int_0^t\La_{\al_s}\e^{-\int_s^t\mu_{\al_r}\d r}\d s,\quad
    a.s.
\end{equation}
\end{lem}

\begin{proof}
The existence and uniqueness of positive solutions to \eqref{1.1}
and \eqref{v4} is more or less standard via a piecewise
deterministic approach. Whereas, we herein provide a sketch of the
proof to make the content self-contained.

Denote $0=\tau_0<\tau_1<\tau_2<\cdots<\tau_n<\cdots$ by the
collection of all jump times of the  Markov chain $(\al_t)_{t\ge0}$.
For any $t\in[0,\tau_1) $, under the assumption ({\bf A1}),
\eqref{1.1} with $\al_t\equiv\al_0$ has a unique strong solution
$(S_t,I_t,R_t)\in\R^3_+$ by exploiting the Lyapunov function, for an
appropriate constant $a>0$,
\begin{equation*}
V(x)=x_1-a-a\ln (x_1/a)+x_2-1-\ln x_2+x_3-1-\ln
x_3>0,~~x=(x_1,x_2,x_3)\in\R^3_+
\end{equation*}
due to $y-1-\ln y>0$ for any $y>0.$ In detail, please refer to the
argument of e.g.
 \cite[Theorem
3.1]{YM}). Next, for any $t\in[\tau_1,\tau_2),$ under the assumption
({\bf A1}), \eqref{1.1} with $\al_t\equiv\alpha_{\tau_1}$ also
admits a unique positive solution by adopting the same test function
$V(x)$  above. Duplicating the previous procedure, we come to a
conclusion that \eqref{1.1} enjoys a unique positive solution as for
the initial value $(s_0,i_0,r_0)\in\R^3_+$.


Next, we aim  to  verify \eqref{2.2}.
 From  \eqref{1.1}, we arrive at
 \begin{equation}\label{eq8}
   \d N_t =\{\La_{\al_t}-\mu_{\al_t}N_t-\delta_{\al_t}I_t\}\d t,~~~~~t>0,
  \end{equation}
  which, along with $I_t\ge0 $, implies  that
  \begin{equation*}
 \d N_t\le \{\La_{\al_t}-\mu_{\al_t}N_t\}\d t.
  \end{equation*}
This enables particularly us to obtain that
\begin{equation*}
\d N_t\le \{\La_{\al_{\tau_k}}-\mu_{\al_{\tau_k}}N_t\}\d t,
~~~t\in[\tau_k,\tau_{k+1}),~~~k\in\mathbb{N}.
\end{equation*}
Subsequently,   the chain rule  yields inductively that
\begin{align*}
N_t&\le\e^{-\mu_{\al_{\tau_k}}(t-\tau_k)}N_{\tau_k}+\int_{\tau_k}^t\La_{\al_{\tau_k}}
\e^{-\mu_{\al_{\tau_k}}(t-s)}\d
s\\
&=\e^{-\int_{\tau_k}^t\mu_{\al_s}\d
s}N_{\tau_k}+\int_{\tau_k}^t\La_{\al_s}\e^{-\int_s^t\mu_{\al_u}\d
u}\d s\\
&\le\e^{-\int_{\tau_k}^t\mu_{\al_s}\d
s}\Big\{\e^{-\int_{\tau_{k-1}}^{\tau_k}\mu_{\al_s}\d
s}N_{\tau_{k-1}}+\int_{\tau_{k-1}}^{\tau_k}\La_{\al_s}\e^{-\int_s^{\tau_k}\mu_{\al_u}\d
u}\d s\Big\}+\int_{\tau_k}^t\La_{\al_s}\e^{-\int_s^t\mu_{\al_u}\d
u}\d s\\
&=\e^{-\int_{\tau_{k-1}}^t\mu_{\al_s}\d
s}N_{\tau_{k-1}}+\int_{\tau_{k-1}}^t\La_{\al_s}\e^{-\int_s^t\mu_{\al_u}\d
u}\d s\\
&\le\cdots\\
&\le\e^{-\int_0^t\mu_{\al_s}\d
s}N_0+\int_0^t\La_{\al_s}\e^{-\int_s^t\mu_{\al_u}\d u}\d s.
\end{align*}
Whence, \eqref{2.2} is now available.
\end{proof}

\begin{rem}\label{re2}
It seems that the assumption ({\bf A1}) excludes the setting on
unbounded incidence function. Concerning such setup, to verify the
positive property of the solutions to \eqref{1.1} and \eqref{v4}
$(\mbox{or } \eqref{jump})$, it is sufficient to follow  the
argument of Lemma \ref{t2.1} and combine with the cut-off approach.
So Lemma \ref{t2.1} still holds whenever the assumption ({\bf A1})
is replaced by ({\bf A1'}) below
\begin{enumerate}
\item[({\bf A1'})] For each $i\in\S$, $G(\cdot,i):\R\rightarrow\R_+$
is locally Lipschtz continuous and that  there exist  constants
$c,k>0$ such that $ G(x,i)\le c\,(1+|x|^k),x\in\R.$
\end{enumerate}

\end{rem}

As a byproduct of Lemma \ref{t2.1}, we derive that
\begin{cor}
Under the assumption $({\bf A1})$,  $(S_t,I_t,R_t,\al_t)_{t\ge0}$
admits an invariant probability measure.
\end{cor}
\begin{proof}
Remark that $(S_t,I_t,R_t,\al_t)_{t\ge0}$ is a Feller process.
According to \eqref{2.2}, we deduce that
\begin{equation}\label{eq3}
S_t\leq N_t\leq N_0\e^{-\hat\mu t}+\check\La/\hat\mu.
\end{equation}
For any $R>0$, let $B_R(0)=\{(s_0,i_0,r_0)\in\R^3_+:s_0+i_0+r_0\le
R\}$ and $P_t(s_0,i_0,r_0,i;\cdot)$ be the transition kernel of
$(S_t,I_t,R_t,\al_t)$ with the starting point
$(s_0,i_0,r_0,i)\in\R^3_+\times\mathcal {M}$. For any $t>0$ and
$\Gamma\in\mathscr{B}(\R^3_+\times\mathcal {M})$, define the
probability measure
\begin{equation*}
\mu_t(\Gamma)=\frac{1}{t}\int_0^tP_s(s_0,i_0,r_0,i;\Gamma)\d s.
\end{equation*}
Then, for any $\varepsilon>0$,  by means of Chebyshev's inequality
and \eqref{eq3}, there exists an $R>0$ sufficiently large such that
\begin{equation*}
\mu_t(B_R(0)\times\mathcal
{M})=\frac{1}{t}\int_0^tP_s(s_0,i_0,r_0,i;B_R(0)\times\mathcal
{M})\d s\ge1-\frac{1}{R}\sup_{t\ge0}\E N_t\ge 1-\varepsilon.
\end{equation*}
Hence, $(\mu_t)_{t\ge0}$ is tight since $B_R(0)$ is a compact subset
of $\R^3_+.$    As a result, $(S_t,I_t,R_t,\al_t)_{t\ge0}$ admits an
invariant probability measure via Krylov-Bogoliubov's theorem (see
e.g. \cite[Theorem 3.1.1]{DZ}).
\end{proof}

Our first main result in this paper is stated as below.
\begin{thm}\label{t3.1}
Suppose $({\bf A1})$ and   $({\bf A2})$  hold  and assume further
 that there exist $\Phi:[0,\infty)\rightarrow[0,\infty)$ with
$\lim_{t\rightarrow\infty}\Phi_t=0$ and
$\Upsilon:\S\rightarrow[0,\infty)$ such that
\begin{equation}\label{c1}
G(I_t,\al_t)S_t/I_t\le \Phi_t+\Upsilon_{\al_t}
\end{equation}
and that
\begin{equation}\label{a5}
\Theta_0:=\frac{\sum_{i\in\S}\pi_i\Upsilon_i}{\sum_{i\in
\S}\pi_i(\mu_i+\nu_i+\delta_i)} <1.
\end{equation}
Then
\begin{equation}\label{q1}
  \lim_{t\ra\infty} I_t=0,\quad a.s.~~~~\mbox{ and }~~~
  \lim_{t\ra \infty} R_t=0,\quad a.s.
\end{equation}
and
\begin{equation}\label{k1}
 \lim_{t\ra\infty}\Big(\frac{1}{t}\int_0^t\mu_{\al_s}S_s\d
 s\Big)=\sum_{i\in\S}\pi_i\Lambda_i.
\end{equation}

\end{thm}

\begin{proof}
Keep $(S_t,I_t,R_t)\in\R_+^3$ in mind due to Lemma \ref{t2.1}. From
\eqref{1.1} and \eqref{c1}, it follows that
\begin{equation}\label{v1}
\begin{split}
  \frac{\d }{\d t}\ln I_t&=G(I_t,\al_t) S_t/I_t-(\mu_{\al_t}+\nu_{\al_t}+\delta_{\al_t})\\
  &\leq \Phi_t+\Upsilon_{\al_t}-(\mu_{\al_t}+\nu_{\al_t}+\delta_{\al_t}).
\end{split}
\end{equation}
So one has
\begin{equation}\label{d1}
\begin{split}
\ln (I_t/I_0)
 &\le \int_0^t
 \Phi_s\d s+\int_0^t\{ \Upsilon_{\al_s}-(\mu_{\al_s}+\nu_{\al_s}+\delta_{\al_s})
 \}\d s.
\end{split}
 \end{equation}
Hence, by virtue of the strong ergodicity theorem for Markov chains
(see  e.g.  \cite{And}), besides
$\lim_{t\rightarrow\infty}\Phi_t=0$, we arrive at
\begin{equation*}
\limsup_{t\ra \infty}\frac{\ln I_t}{t}\leq
\sum_{i\in\S}\{\Upsilon_i-(\mu_i+\nu_i+\delta_i)\}\pi_i,\quad a.s.
\end{equation*}
Thus, $\lim_{t\ra\infty} I_t=0$, a.s., follows from \eqref{a5}.

In what follows, we intend to show $\lim_{t\ra \infty} R_t=0$,\
a.s. To end this, observe that
\[ \d R_t \leq (\check\nu I_t-(\hat\mu+\hat\gamma)R_t)\d t.\]
Subsequently, by applying the chain rule to
$\d(\e^{(\hat\mu+\hat\gamma)t}R_t)$, we deduce that
\begin{equation}\label{eq0}
R_t\leq R_0\e^{-(\hat\mu+\hat\gamma)
t}+\check\nu\int_0^t\e^{-(\hat\mu+\hat\gamma)(t-s)}I_s\d s.
\end{equation}
Since $\lim_{t\ra \infty} I_t=0$,\ a.s.,     for any
$\varepsilon>0$, there exist $\Omega_0\subseteq\Omega$ with
$\mathbb{P}(\Omega_0)=1$ and $T=T(\omega)>0$ such that
\begin{equation*}
I_t(\omega)\le
\varepsilon(\hat\mu+\hat\gamma)/(3\check\nu),~~~~~t\ge
T,~~~\omega\in\Omega_0,
\end{equation*}
which of course implies that
\begin{equation*}
\check\nu\int_T^t\e^{-(\hat\mu+\hat\gamma)(t-s)}I_s(\omega)\d s\le
\varepsilon/3,~~~~\omega\in\Omega_0,~~~t\ge T.
\end{equation*}
This, in addition to \eqref{eq0}, yields that
\begin{equation*}
\begin{split}
R_t(\omega)&\leq R_0\e^{-(\hat\mu+\hat\gamma)
t}+\check\nu\int_0^T\e^{-(\hat\mu+\hat\gamma)(t-s)}I_s(\omega)\d
s+\check\nu\int_T^t\e^{-(\hat\mu+\hat\gamma)(t-s)}I_s(\omega)\d s\\
&\le \varepsilon/3+R_0\e^{-(\hat\mu+\hat\gamma)
t}+\check\nu(N_0+\check\La/\hat\mu)\int_0^T\e^{-(\hat\mu+\hat\gamma)(t-s)}\d
s\\
&\le\varepsilon,~~~~~\omega\in\Omega_0
\end{split}
\end{equation*}
for  any $$t\ge
T\vee\Big(\frac{1}{\hat\mu+\hat\gamma}\ln\frac{3R_0}{\varepsilon}\Big)\vee
\Big(T+\frac{1}{\hat\mu+\hat\gamma}\ln\frac{3\check\nu(N_0+\frac{\check\La}{\hat\mu})}{\varepsilon(\hat\mu+\hat\gamma)}\Big).$$
Consequently, $\lim_{t\ra \infty} R_t=0$,\  a.s., follows
immediately.

By \eqref{q1}, one has
\begin{equation}\label{c00}
 \lim_{t\ra\infty}  \Big(\frac{1}{t}\int_0^tI_s\d s\Big)=0,\quad a.s.~~~~\mbox{ and }~~~
  \lim_{t\ra \infty}\Big(\frac{1}{t}\int_0^t R_s\d s\Big)=0,\quad a.s.
\end{equation}
From \eqref{eq8}, it follows that
\begin{equation*}
  \frac{1}{t}\int_0^t\mu_{\al_s}S_s\d s=\frac{ N_0-N_t}{t} +\frac{1}{t}\int_0^t\{\La_{\al_s}-\mu_{\al_s}R_s-(\mu_{\al_s}+\delta_{\al_s})I_s\}\d s,~~~~~t>0.
  \end{equation*}
This, in addition to \eqref{eq3}, \eqref{c00} as well as  the strong
ergodic theorem for the continuous-time Markov chains, yields the
assertion \eqref{k1}.
\end{proof}

It is easy to examine that all the incidence rate functions with
$\ell>1$ satisfy \eqref{c1} by taking advantage of Lemma \ref{t2.1}.
Now we present some applications of Theorem \ref{t3.1}. Firstly, in
\eqref{1.1} we choose
\begin{equation}\label{c6}
G(x,i)= \beta_ix/f(x),~~~~x\ge0,
\end{equation} where
$\beta_\cdot:\S\rightarrow\R_+$ and
\begin{enumerate}
\item[({\bf A3})]
 $f:\R_+\rightarrow\R_+$ with $f(0)=1$ is continuous  and $f'(x)>0$ for any $ x\ge0.$
\end{enumerate}

The corollary  below provides a sufficient criterion to examine the
extinction of the infectious individuals even though the infectious
individuals are persistent in some fixed environments.

\begin{cor}\label{co1}
Let $({\bf A3})$  hold and assume that
\begin{equation}\label{eq4}
\Theta_1:=\frac{\check \La \sum_{i\in\S}\pi_i\beta_i}{\hat\mu
\sum_{i\in\S}\pi_i(\mu_i+ \nu_i+\delta_i)}<1.
\end{equation}
Then, for $(S_t,I_t,R_t)_{t\ge0}$ solved by \eqref{1.1} and
\eqref{v4} with $G$ in \eqref{c6}, all of the assertions in Theorem
\ref{t3.1} hold.

\end{cor}

\begin{proof}
By $f(0)=1$ and $f'(x)>0$ for any $ x\ge0$, we deduce that $G(x,i)=
\beta_ix/f(x)$ satisfies the assumption ({\bf A1}) so that Lemma
\ref{t2.1} is applicable. From \eqref{eq3}, together with  $f(0)=1$
and $f'(x)>0$ for any $ x\ge0$, we
  obtain   that
\begin{equation*}
\begin{split}
G(I_t,\al_t)S_t/I_t&=\beta_{\al_t}S_t/f(I_t)\le\beta_{\al_t}N_t
\le\beta_{\al_t}(N_0\e^{-\hat\mu t}+\check\La/\hat\mu).
\end{split}
\end{equation*}
As a consequence, we infer that \eqref{c1} holds with
$\Phi_t=c\,\e^{-\hat\mu t}$ for some  $c>0$ and $\Upsilon_{\al_t}=
\check\La\beta_{\al_t}/\hat\mu $ so that \eqref{a5} is satisfied
thanks to \eqref{eq4}. Thus, the desired assertions  follow  from
Theorem \ref{t3.1}.
\end{proof}

Another application of Theorem \ref{t3.1} is to take
\begin{equation}\label{oo}
G(x,i)=\frac{\beta_ix^i}{1+ax^2},~~~~x\ge0
\end{equation} for some
$\beta_\cdot:\S\rightarrow\R_+$ and $a>0$. The following corollary
reveals   the influence of the random switching of the incidence
functions on
 extinction of infectious individuals.
\begin{cor}\label{co}
Let $({\bf A2})$  hold and assume that
\begin{equation}\label{c3}
\Theta_2:=\frac{\sum_{i\in\S}\pi_i\beta_i\big(\check\La/\hat\mu\big)^i}{\sum_{i\in
\S}\pi_i(\mu_i+\nu_i+\delta_i)} <1,
\end{equation}
then, for $(S_t,I_t,R_t)_{t\ge0}$ solved by \eqref{1.1} and
\eqref{v4} with $G$ in \eqref{oo}, all of the assertions in Theorem
\ref{t3.1} hold.
\end{cor}

\begin{proof}
In terms of the definition of $G$ introduced in \eqref{oo} and by
taking Remark \ref{re2} into account, we deduce that
 $0<I_t,S_t\le N_t$ and \eqref{eq3} holds so that
\begin{equation}\label{a2}
\begin{split}
G(I_t,\al_t)S_t/I_t&=\frac{\beta_{\al_t}I_t^{\al_t-1}S_t}{1+aI_t^2}\le\beta_{\al_t}N_t^{\al_t}=\beta_{\al_t}\sum_{i=1}^MN_t^{i}{\bf
1}_{\{\al_t=i\}}\\
&\le \beta_{\al_t}\sum_{i=1}^M(N_0\e^{-\hat\mu
t}+\check\La/\hat\mu)^i{\bf 1}_{\{\al_t=i\}}\\
&\le c\,\e^{-\hat\mu
t}+\beta_{\al_t}\sum_{i=1}^M(\check\La/\hat\mu)^i{\bf
1}_{\{\al_t=i\}}\\
&=c\,\e^{-\hat\mu t}+\beta_{\al_t} (\check\La/\hat\mu)^{\al_t}
\end{split}
\end{equation}
for some constant $c>0.$ Therefore, \eqref{c1} holds with
$\Phi_t=c\,\e^{-\hat\mu t}$ and $\Upsilon_{\al_t}=\beta_{\al_t}
(\check\La/\hat\mu)^{\al_t}$. On the other hand, \eqref{a5} follows
owing to \eqref{c3}. As a result, all of the assertions  hold  true
in terms of Theorem \ref{t3.1}.
\end{proof}

Now we proceed to provide some examples to illustrate the
applications of Corollaries \ref{co1} and \ref{co} so that our main
result (i.e. Theorem \ref{t3.1}) is applicable. To portray the
behavior of the infectious individuals in each fixed environment, we
introduce the quantity $R_0^{(i)}, i\in\S,$ defined by
\begin{equation}\label{thres-1}R_0^{(i)}=\frac{\La_i\beta_i}{\mu_i
(\mu_i+\nu_i+\delta_i)},~~~i\in\S.\end{equation}

\begin{exam}\label{exam1}
{\rm We focus on
 the model \eqref{1.1}, in which   $G$ is given as in \eqref{c6} with   $f(0)=1$, and $(\alpha_t)_{t\ge0}$ is a
continuous-time Markov chain with the state space $\mathcal
{M}=\{1,2\}$ and the $Q$-matrix
\begin{equation}\label{z1}
Q= \left(\begin{array}{ccc}
  -p & p\\
  q & -q\\
  \end{array}
  \right),~~~~p,~q>0.
  \end{equation}
 Assume that
\begin{equation}\label{w00}
\La_2<\La_1,~~~~~\mu_2<\mu_1,
\end{equation}
\begin{equation}\label{w111}
(\mu_1+\nu_1+\delta_1)/\beta_1> \La_1/\mu_2,~~~
(\mu_2+\nu_2+\delta_2)/\beta_2< \La_2/\mu_2,
\end{equation}
and that
\begin{equation}\label{w3}
\frac{q}{p}>\frac{\La_1\beta_2-\mu_2(\mu_2+\nu_2+\delta_2)}{\mu_2(\mu_1+\nu_1+\delta_1)-\La_1\beta_1}>0.
\end{equation}}
\end{exam}
It is apparent that \eqref{w00}
 implies that  $\check
\La=\La_1$ and $\hat\mu=\mu_2$. By a simple calculation, the unique
invariant probability measure of $(\alpha_t)_{t\ge0}$ is given by
\begin{equation}\label{a3}
\pi=(\pi_1,\pi_2)=\Big(\frac{q}{p+q},\frac{p}{p+q}\Big).
\end{equation}
Hence, by taking  \eqref{w3} into consideration, it follows that
\begin{equation*}
\Theta_1=\frac{\La_1 (q\beta_1+p\beta_2)}{\mu_2\{q(\mu_1+
\nu_1+\delta_1)+p(\mu_2+ \nu_2+\delta_2)\} }<1.
\end{equation*}
Whence, Corollary \ref{co1} implies  $\lim_{t\ra \infty} I_t=0$,
a.s., and, $\lim_{t\ra \infty} R_t=0$,  a.s.

In view of  $f(0)=1$ and
  \eqref{w111}, one has
$ R_0^{(1)}<1, R_0^{(2)}>1$. As a result, in terms of \cite[Theorem
2.1]{Cai15}, the disease-free equilibrium
$E_0^{(1)}:=(\La_1/\mu_1,0,0)$ of the SIRS model \eqref{1.1} with
$\al_t\equiv1$
  is globally
asymptotically stable,  whereas   the disease-free equilibrium
$E_0^{(2)}:=(\La_2/\mu_2,0,0)$ of the SIRS model \eqref{1.1} with
$\al_t\equiv2$
  is unstable.
Obviously, \eqref{w111}  and \eqref{w3} hold  for
\begin{equation*}
\beta_1=\mu_2=\nu_2=\Lambda_2=0.1,~\mu_1=\nu_1=p=0.2,~\beta_2=0.3,~\delta_1=0.05,~\delta_2=0.04,~\Lambda_1=0.4,q=4.
\end{equation*}

The following example shows that the random switching of the
incidence functions can improve the extinction of the infectious
individuals in certain sense. More precisely, for the model
\eqref{1.1} with $G$ given in \eqref{oo}, the infectious individuals
are extinct although they might be persistent with certain incidence
function in some environment.

\begin{exam}\label{exam2}
{\rm Consider the model \eqref{1.1} with $G$ introduced in
\eqref{oo}, where $(\alpha_t)$ is a continuous-time Markov chain
with the state space $\S=\{1,2\}$ and the $Q$-matrix $Q$ given by
\eqref{z1}. Assume that
\begin{equation}\label{r1}
\La_2<\La_1, ~~\mu_2<\mu_1,
\end{equation}
\begin{equation}\label{r3}
\beta_2\Big(\frac{\La_1}{\mu_2}\Big)^2<\mu_2+\nu_2+\delta_2,~~~
\mu_1(\mu_1+\nu_1+\delta_1)<\beta_1\La_1
\end{equation}
and that
\begin{equation}\label{w7}
0<\frac{q}{p}<\frac{\mu_2+\nu_2+\delta_2-\beta_2\Big(\frac{\La_1}{\mu_2}\Big)^2}{\frac{\beta_1\La_1}{\mu_2}-(\mu_1+\nu_1+\delta_1)}.
\end{equation}
}
\end{exam}

Thanks to \eqref{r1}, one has $\check \La=\La_1$ and
$\hat\mu=\mu_2.$ On the other hand, it is easy to see that
$\Theta_2$, defined in \eqref{c3}, is less than $1$ by virtue of
\eqref{w7}. Hence, Corollary \ref{co}  demonstrates that
$\lim_{t\ra\infty} I_t=0$, a.s.,  and,
  $\lim_{t\ra \infty} R_t=0$, a.s.

  Also, it follows from \eqref{r3}
  that $R_0^{(2)}$, defined in \eqref{thres-1}, is greater than $1$.
Consequently, $I_t$ with $\al_t\equiv1$ is unstable due to
\cite[Theorem 2.1]{Cai15}. More concretely, by taking
\begin{equation*}
\begin{cases}
\La_1=0.4\\
\La_2=0.3
\end{cases}
\begin{cases}
\mu_1=0.25\\
\mu_2=0.2
\end{cases}
\begin{cases}
\beta_1=0.3\\
\beta_2=0.1
\end{cases}
\begin{cases}
\nu_1=0.2\\ \nu_2=0.3
\end{cases}
\begin{cases}
\delta_1=0.02\\ \delta_2=0.05
\end{cases}
\begin{cases}
 p=3\\ q=2
\end{cases},
\end{equation*}
we find that \eqref{r1}-\eqref{w7} hold, respectively.

The following theorem presents some sufficient conditions to depict
the persistence of the infectious individuals. The criterion
provided allows the infectious individuals to be extinct in certain
environments.

\begin{thm}\label{t2.3}
Let $({\bf A1})$ and $({\bf A2})$    hold and suppose further that
$\lim_{x\rightarrow0}G(x,j)/x>0$ for any $j\in\S$ and that
\begin{equation}\label{cutoff2}
\Theta_3:=\frac{\sum_{i\in\S}\pi_i\La_i}
{\tau\sum_{i\in\S}\pi_i(\mu_i+\nu_i+\delta_i)}>1,
\end{equation}
where
\begin{equation}\label{d2}
\tau:=\max_{j\in\S}\Big(\frac{G(0,j)+\mu_j}{\lim_{x\rightarrow0}G(x,j)/x}\Big).
\end{equation}
Then
\begin{equation}\label{2.9}
\liminf_{t\rightarrow\infty}\Big(\frac{1}{t}\int_0^t I_s\d
s\Big)>0,\quad a.s.,
\end{equation} that is, the infectious individuals is persistent.
\end{thm}

\begin{proof}
 First of all, we claim that there exists a constant $K>0$ such
that
\begin{equation}\label{2.10}
F_{j,y}(x):= K x+((\tau/x-1)G(x,j) -\mu_j)y\geq 0
\end{equation}
for any $0\le x,y\le N_0+\check\La/\hat\mu$ and $j\in\S.$ Obviously,
\eqref{2.10} holds whenever $y=0.$ So, it is sufficient to verify
that \eqref{2.10} holds for $0<y\le N_0+\check\La/\hat\mu$. In what
follows, we set $0<y\le N_0+\check\La/\hat\mu$. According to the
definition of $\tau$,  it is obvious to see that
\begin{equation}\label{c2}
F_{j,y}(0)=(\tau \lim_{x\rightarrow0}G(x,j)/x-G(0,j) -\mu_j)y>0.
\end{equation}
By  the continuity of $x\mapsto F_{j,y}(x)$, we deduce from
\eqref{c2} that there exists $0<x_0<N_0+\check\La/\hat\mu$ such that
\eqref{2.10} holds for some $K=K_0>0$ and any $x\in[0,x_0].$ Next,
for any $x\in[x_0,N_0+\check\La/\hat\mu]$, observe that
\begin{equation}
\begin{split}
F_{j,y}(x)\ge K
x_0-\max_{x\in[x_0,N_0+\check\La/\hat\mu],j\in\S}|(\tau/x-1)G(x,j)
-\mu_j|(N_0+\check\La/\hat\mu).
\end{split}
\end{equation}
Thus, \eqref{2.10} is available  by taking $K>0$ sufficiently large.

 Taking advantage
of  $R_t\ge0$, a.s., we infer that
\begin{align*}
  \frac{\d }{\d
  t}S_t&\ge\La_{\al_t}-\mu_{\al_t}S_t-G(I_t,\al_t)S_t\\
  &=\La_{\al_t}-\tau G(I_t,\al_t)S_t/I_t-KI_t+F_{\al_t,S_t}(I_t)\\
  &\ge \La_{\al_t}-\tau G(I_t,\al_t)S_t/I_t-KI_t,
\end{align*}
which further yields that
\begin{equation}\label{2.11}
\tau \int_0^t\frac{G(I_s,\al_s)S_s}{I_s}\d s\geq S_0-S_t- K\int_0^t
I_s\d s+\int_0^t\La_{\al_s}\d s.
\end{equation}
Substituting the first display of \eqref{v1}  into \eqref{2.11} and
taking \eqref{eq3} into account, one has
\begin{align}\label{d3}
K\int_0^t I_s\d s&\geq S_0-N_0-\check\La/\hat\mu- \tau\ln c
+\int_0^t( \La_{\al_s}-\tau
  (\mu_{\al_s}+\nu_{\al_s}+\nu_{\al_s}))\d s
\end{align}
for some constant $c>0$. Consequently, the strong ergodicity theorem
for the Markov chain $(\alpha_t)_{t\ge0}$  yields that
\begin{align*}
 K\liminf_{t\rightarrow\infty}\Big(\frac{1}{t}\int_0^t I_s\d s\Big)\ge&\sum_{j\in\mathcal {M}}(
 \La_j-\tau
  (\mu_j+\nu_j+\delta_j))\pi_j.
\end{align*}
Whence, \eqref{2.9} follows directly from \eqref{cutoff2}.
\end{proof}

Applying Theorem \ref{t2.3} to the incidence rate function $G$ in
\eqref{c6}, we obtain the following corollary, which  states some
sufficient conditions to examine the persistence of the individuals.

\begin{cor}\label{co2}
Let $({\bf A2})$   hold and suppose further that
\begin{equation}\label{v2}
\Theta_4:=\frac{\sum_{i\in\S}\pi_i\La_i} {
\max_{i\in\S}(\mu_i/\beta_i)\sum_{i\in\S}\pi_i(\mu_i+\nu_i+\delta_i)}>1.
\end{equation}
Then
\begin{equation}\label{v3}
\liminf_{t\rightarrow\infty}\Big(\frac{1}{t}\int_0^t I_s\d
s\Big)>0,\quad a.s.
\end{equation}
\end{cor}

\begin{proof}
By the structure of $G$ given in \eqref{c6}, we have $G(0,j)=0$ and
$\lim_{x\rightarrow0}G(x,j)/x=\beta_j$ due to $f(0)=1$ so that
$\tau= \max_{i\in\S}(\mu_i/\beta_i).$ With \eqref{v2} in hand, we
therefore infer that \eqref{cutoff2} holds. Accordingly, the desired
assertion \eqref{v3} is verified.
\end{proof}

Below, let's revisit Example \ref{exam1} which, under certain
appropriate conditions, illustrates that the infectious individuals
is persistent although they might die out in some environments.

\begin{exam}\label{ex2}
{\rm Let's reconsider Example \ref{exam1}. Assume that
\begin{equation}\label{a1}
 \frac{\mu_2+\nu_2+\delta_2}{\La_2}<\frac{\beta_1}{\mu_1}<\frac{\beta_2}{\mu_2}\wedge
 \frac{\mu_1+\nu_1+\delta_1}{\La_1},
\end{equation}
and that
\begin{equation*}
0<\frac{q}{p}<\frac{\beta_1\La_2-\mu_1(\mu_2+\nu_2+\delta_2)}{\mu_1(\mu_1+\nu_1+\delta_1)-\beta_1\La_1}.
\end{equation*}

In accordance with \eqref{a3} and \eqref{a1}, $\Theta_3$, introduced
in \eqref{cutoff2}, reads as follows
\begin{equation*}
\Theta_3=\frac{q\La_1+p\La_2}
{\frac{\mu_1}{\beta_1}\{q(\mu_1+\nu_1+\delta_1)+p(\mu_2+\nu_2+\delta_2)\}}<1.
\end{equation*}
 Thus, with the help of Corollary  \ref{co2}, we arrive at
$$\liminf_{t\rightarrow\infty}\Big(\frac{1}{t}\int_0^t I_s\d
s\Big)>0,~~~~\mbox{ a.s.}$$ Nevertheless, by virtue of  \eqref{a1},
it follows that  $R_0^{(1)}<1$ and $R_0^{(2)}>1$ such that
$\lim_{t\rightarrow\infty}I_t=0$ (for the case $\al_t\equiv1$),
a.s., and $\lim_{t\rightarrow\infty}I_t>0$ (for the case
$\al_t\equiv2$), a.s. So, the infectious individuals is persistent
in the environment $1$ and dies out in the environment $2$. }
\end{exam}

The corollary below explicates that the assumptions imposed in
Corollaries \ref{co1} and \ref{co2} are compatible.

\begin{cor}
  It holds $\Theta_4\leq \Theta_1$, and further $\Theta_4=\Theta_1$ if and only if $\La_i$, $\beta_i$, $\mu_i$ are all independent of $i$, i.e., for some positive constants $\La,\,\beta,\,\mu$, $\La_i\equiv \La$, $\beta_i\equiv \beta$, $\mu_i\equiv \mu$, $i\in\S$.
\end{cor}
\begin{proof}
 According to the notions of $\Theta_1$ and $\Theta_4$, we deduce
 that
 \begin{equation}\label{y1-1}
 \begin{split}
   \frac{\Theta_4}{\Theta_1}&=\frac{\sum_{i\in\S} \pi_i\big(\frac{\La_i}{\check \La}\big)}{\sum_{i\in\S}\pi_i\beta_i\max_{j\in\S}\big(\frac{\mu_j}{
   \hat\mu}\cdot\frac 1{\beta_j}\big)}\le  \frac{\sum_{i\in\S} \pi_i\big(\frac{\La_i}{\check \La}\big)}{\sum_{i\in\S}\pi_i\big(\frac{\beta_i}{\hat\beta
   }\big)}\le1,
   \end{split}
 \end{equation}
that is, $\Theta_4\leq \Theta_1$. It is obvious to observe that
$\Theta_4=\Theta_1$ whenever $\La_i, \beta_i$ and $\mu_i$ are
constant. Now, if   $\Theta_4= \Theta_1$, in view of the first
inequality in \eqref{y1-1}, we have $\check\La=\La_i$ and
$\hat\beta=\beta_i, i\in\S$, namely, both $\La_i$ and $\beta_i$ are
constant. Whence, exploiting the identity in \eqref{y1}, we arrive
at $\check \beta=\hat\beta$, which further means that $\beta_i,
i\in\S,$ is independent of the index $i$.
\end{proof}

In Examples \ref{exam1} and \ref{ex2}, we explain that the
infectious individuals are extinct (resp. persistent) although they
might persist (resp. die out) in some environments. Yet one may be
quite interested in the examples, where the infectious individuals
are extinct (resp. persistent) even though the infectious
individuals are persistent (resp. extinct) in each fixed
environment. Nevertheless, the following corollary shows that the
  scenario  mentioned cannot take place.

\begin{cor} Let $({\bf A2})$   hold.
\begin{itemize}
  \item[$(i)$] If, for each $i\in\S$,   $R_0^{(i)}\leq 1$, then it always holds $\Theta_4\leq 1$ whatever the irreducible transition
  rate matrix of the random switching process $(\alpha_t)_{t\ge0}$ is.
  \item[$(ii)$] If, for each $i\in \S$, $R_0^{(i)}>1$, then it always hold $\Theta_4>1$ whatever the irreducible transition rate matrix of
  the random switching process $(\alpha_t)_{t\ge0}$ is.
\end{itemize}
\end{cor}
\begin{proof}
  By the definition of $R_0^{(i)}$ introduced in \eqref{thres-1}, one has $\mu_i+\nu_i+\delta_i=\frac{\La_i\beta_i}{\mu_i R_0^{(i)}}$.
   Then, $\Theta_0$ and $\Theta_1$ can be reformulated, respectively, as
  \[\Theta_1= \frac{\check \La\sum_{i\in\S}\pi_i\beta_i}{\hat\mu \sum_{i\in\S}\pi_i\frac{\La_i\beta_i}{\mu_i R_0^{(i)}}}~~
\mbox{ and }~~
\Theta_4=\frac{\sum_{i\in\S}\pi_i\La_i}{\max_{i\in\S}\big(\frac{\mu_i}{\beta_i}\big)
  \sum_{i\in\S}\pi_i\frac{\La_i\beta_i}{\mu_i R_0^{(i)}}}. \]
  If $R^{(i)}_0\leq 1$ for each $i\in \S$, then
  \[\Theta_4=\frac{\sum_{i\in\S}\pi_i\La_i}{\sum_{i\in \S}\pi_i\La_i\cdot\frac{\beta_i/\mu_i}{\min_{j\in\S}
  (\beta_j/\mu_j)}\cdot \frac{1}{R_0^{(i)}}}\leq 1,\] owing to  $\frac{\beta_i/\mu_i}{\min_{j\in\S}(\beta_j/\mu_j)}\geq 1$ and $1/R_0^{(i)}\geq 1$.
  This gives us the assertion (i). Next,
in case of $R^{(i)}_0>1$, it follows from $\La_i/\check\La\leq 1$
and $\hat\mu/\mu_i\leq 1$  that
  \[\Theta_1=\frac{\sum_{i\in\S}\pi_i\beta_i}{\sum_{i\in \S}\pi\beta_i\frac{\La_i}{\check \La}\cdot\frac{\hat\mu}{\mu_i}\cdot\frac 1{R_0^{(i)}}}>1,\]
  which further yields  the desired conclusion (ii).
\end{proof}

\section{Impacts of state-dependent  random environments}\label{sec4}
In this section, we move forward to deal with impacts of
state-dependent random environments upon extinction and persistence
of the infectious individuals. As an illustrative work, in this part
we are interested in the SIRS model \eqref{1.1} and \eqref{jump}. As
we know, \eqref{1.1} and \eqref{jump} is a kind of state-dependent
regime switching diffusions, which have been investigated
considerably on e.g. stability, ergodicity and numerical
approximation in the past decade; see e.g. \cite{Shao,XS,YZ} and
references therein. It is worthy to point out that the quadruple
$(S_t,I_t,R_t,\al_t)_{t\ge0}$ is a Markov process although neither
$(S_t,I_t,R_t)_{t\ge0}$ nor $(\al_t)_{t\ge0}$ is. Assume further
that
\begin{itemize}
   \item[({\bf Q1})] For each $x\in\R_+^3$, the matrix $Q(x)=(q_{ij}(x))_{i,j\in\S}$ is irreducible and conservative,
   \item[({\bf Q2})] $H:=\sup_{x\in \R_+^3} \max_{i\in\S}
   q_i(x)<\infty$, where $q_i(x):=\sum_{j\neq i}q_{ij}(x)$.
 \end{itemize}

In contrast to the SIRS model \eqref{1.1} and \eqref{v4}, there are
essential challenges to cope with the model \eqref{1.1} and
\eqref{jump}. For this setup, one of the challenges is that the
classical ergodic theorem for continuous-time Markov chains does not
work any more due to the fact that $(\al_t)_{t\ge0}$ is merely a
jump process rather than a Markov process. To get over such
difficulty, we shall employ a stochastic comparison  for functionals
of the jump process $(\alpha_t)_{t\ge0}$. More precisely,
\begin{lem}\label{lm}
Assume $({\bf Q1})$ and $({\bf Q2})$ hold, and further $q_{ij}(x)=0$
for every $i,\,j\in\S$ with $|i-j|\geq 2$ and every $x\in\R^3$. For
every $i,\,j\in\S$, let
\begin{equation*}
  q_{ij}^\ast=\begin{cases}
    \sup_{x\in\R^3}q_{ij}(x),\quad &j<i\\
    \inf_{x\in\R^3}q_{ij}(x),\quad &j>i\\
    -\sum_{i\neq j}q_{ij}^\ast, &j=i
  \end{cases} \quad \text{and}\quad
  \bar q_{ij}=\begin{cases}
    \inf_{x\in\R^3}q_{ij}(x),\quad &j<i\\
    \sup_{x\in\R^3}q_{ij}(x),\quad &j>i\\
    -\sum_{j\neq i}\bar q_{ij},  &j=i
  \end{cases}.
\end{equation*}
Suppose that $(q_{ij}^\ast)$ and $(\bar q_{ij})$ are irreducible and
satisfy
\begin{equation}\label{econ-1}
\begin{split}
&q_{i,i+1}(x)+q_{i+1,i}(x) \ \text{is independent of $x$ for $1\leq i\leq N-2$,}\\
& \bar q_{N-1,N}+\bar q_{N,N-1}\leq q_{N-1,N}(x)+q_{N,N-1}(x), \quad \forall\, x\in\R^3,\\
&q_{N-1,N}^\ast+q_{N,N-1}^\ast\geq q_{N-1,N}(x)+q_{N,N-1}(x), \quad
\forall\, x\in\R^3.
\end{split}
\end{equation}
Then, there exist two continuous-time Markov chains
$(\alpha_t^\ast)_{t\ge0}$ and $(\bar\alpha_t)_{t\ge0}$ on $\S$ with
transition rate matrix $(q_{ij}^\ast)$ and $(\bar q_{ij})$
respectively such that for every nondecreasing function $\phi:\S\ra
\R_+$,
\begin{equation}\label{c7}
\int_0^t\phi(\al^*_s)\d s\le \int_0^t\phi(\al_s)\d s\le
\int_0^t\phi(\bar \al_s)\d s,~~~\mbox{ a.s.}
\end{equation}
\end{lem}

\begin{proof}
One can follow the idea of the argument to show \cite[Lemma
2.8]{Shao17} to prove this lemma, although only the upper bound is
proved therein. So the proof of this lemma  is omitted to save
space.
\end{proof}

Under the condition that $(q_{ij}^\ast)$ and $(\bar q_{ij})$ are
irreducible, the finiteness of $\S$ yields that
$(\al_t^\ast)_{t\ge0}$ and $(\bar \al_t)_{t\ge0}$ are positive
recurrent. Let $\pi^*=(\pi_1^*,\pi_2^*,\cdots,\pi_M^*)$ and
$\bar\pi=(\bar\pi_1,\bar\pi_2,\cdots,\bar\pi_M)$ be the invariant
probability measures of the continuous-time Markov chains
$(\al_t^*)_{t\ge0}$ and  $(\bar\al_t)_{t\ge0}$, respectively,
provided that both $(\al_t^*)_{t\ge0}$ and  $(\bar\al_t)_{t\ge0}$
are irreducible and positive recurrent.

As an application of Lemma \ref{lm}, we provide  some sufficient
conditions to judge the extinction of the infectious individuals for
the SIRS model determined by \eqref{1.1} and \eqref{jump}.

\begin{thm}\label{state}
Assume that the assumptions of Theorem \ref{t3.1} and Lemma \ref{lm}
hold.  Suppose further that $i\mapsto\Gamma_i:=
\Upsilon_i-(\mu_i+\nu_i+\delta_i)$ is nondecreasing and that
\begin{equation}\label{c11}
\Theta_5:=\frac{\sum_{i\in\S}\bar\pi_i\Upsilon_i}{\sum_{i\in
\S}\bar\pi_i(\mu_i+\nu_i+\delta_i)} <1;
\end{equation}
or that $i\mapsto\Gamma_i$ is nonincreasing  and that
\begin{equation}\label{c12}
\Theta_6:=\frac{\sum_{i\in\S}\pi_i^*\Upsilon_i}{\sum_{i\in
\S}\pi_i^*(\mu_i+\nu_i+\delta_i)} <1.
\end{equation}
Then \begin{equation}\label{c14}
 \lim_{t\ra \infty} I_t=0, \mbox{ a.s. and }
 \lim_{t\ra\infty} R_t=0, \mbox{ a.s. }
\end{equation}
Moreover, if $i\mapsto\Gamma_i$ is nondecreasing, then
\begin{equation}\label{g1}
\sum_{i\in\S}\pi_i^*\Lambda_i\le
\liminf_{t\ra\infty}\Big(\int_0^t\mu_{\al_s}S_s\d
 s\Big)\le\limsup_{t\ra\infty}\Big(\int_0^t\mu_{\al_s}S_s\d
 s\Big)\le\sum_{i\in\S}\bar\pi_i\Lambda_i;
\end{equation}
or if $i\mapsto\Gamma_i$ is nondecreasing, then
\begin{equation}\label{g2}
\sum_{i\in\S}\bar\pi_i\Lambda_i\le
\liminf_{t\ra\infty}\Big(\int_0^t\mu_{\al_s}S_s\d
 s\Big)\le \limsup_{t\ra\infty}\Big(\int_0^t\mu_{\al_s}S_s\d
 s\Big)\le\sum_{i\in\S}\pi_i^*\Lambda_i.
\end{equation}

\end{thm}

\begin{proof}
Once $\lim_{t\ra \infty} I_t=0$, a.s., is available,
$\lim_{t\ra\infty} R_t=0$, a.s., can be proved similarly by
following the trick in the argument of Theorem  \ref{t3.1}. So, in
what follows, it remains to show that $\lim_{t\ra \infty} I_t=0$,
a.s. Observe that \eqref{d1} still holds for the present setup. So,
taking the nondecreasing property of $i\mapsto\Gamma_i$ as well as
\eqref{c11} and employing Lemma \ref{lm} and the ergodic theorem for
the continuous-time Markov chains, we obtain the desired assertions
\eqref{c14}.

If $i\mapsto\Gamma_i$ is nonincreasing, then  $i\mapsto-\Gamma_i$ is
nondecreasing trivially. So, an application of Lemma \ref{lm}
yields that
\begin{equation*}
-\int_0^t\Gamma_{\al_s^*}\d s\le-\int_0^t\Gamma_{\al_s}\d s\le
-\int_0^t\Gamma_{\bar\al_s}\d s,
\end{equation*}
which further results in
\begin{equation*}
\int_0^t\Gamma_{\bar\al_s}\d s\le\int_0^t\Gamma_{\al_s}\d s\le
\int_0^t\Gamma_{\al_s^*}\d s.
\end{equation*}
Whence, the  assertion \eqref{c14} follows from \eqref{c12} and the
ergodic theorem for the continuous-time Markov chains.

By virtue of \eqref{eq8}, it holds that
 \begin{equation*}
   \frac{1}{t}\int_0^t \mu_{\al_t}S_s\d s=\frac{N_0-N_t}{t} +\frac{1}{t}\int_0^t\{\La_{\al_s} -
   (\mu_{\al_s}+\delta_{\al_s})I_s-\mu_{\al_s}R_s \}\d s.
  \end{equation*}
Thereby, \eqref{g1} and \eqref{g2} follow from \eqref{eq3},
\eqref{c14}, Lemma \ref{lm}, as well as the strong ergodic theorem
for continuous-time Markov chains.
\end{proof}

Hereinafter, two examples are set to show the applications of
Theorem \ref{state}.
\begin{exam}\label{exam3}
{\rm  Consider the model \eqref{1.1} with $G$ introduced in
\eqref{oo}, where $(\alpha_t)_{t\ge0}$ is a continuous-time jump
process with the state space $\S=\{1,2\}$ and the $Q$-matrix $Q(x)$
given by
\begin{equation*}
Q(x)= \bigg(\begin{array}{ccc}
  \sin x-p & p-\sin x\\
  q+\sin x & -q-\sin x\\
  \end{array}
  \bigg).
  \end{equation*}
Assume that
\begin{equation}\label{s4}
\La_2<\La_1, ~~~~\mu_2<\mu_1,
\end{equation}
\begin{equation}\label{s5}
 \beta_1\La_1/\mu_2-(\mu_1+\nu_1+\delta_1)<0,~~~~~~\beta_2( \La_1/\mu_2)^2-(\mu_2+\nu_2+\delta_2)>0
\end{equation}
and that
\begin{equation}\label{s8}
\frac{q-1}{1+p}> \frac{\beta_2(
\La_1/\mu_2)^2-(\mu_2+\nu_2+\delta_2)}{\mu_1+\nu_1+\delta_1-
\beta_1\La_1/\mu_2}.
\end{equation}

 A straightforward calculation shows that
\begin{equation*}
\bar Q=(\bar q_{ij})_{1\le i,j\le 2}= \bigg(\begin{array}{ccc}
   -(p+1) & p+1\\
  q-1 & 1-q\\
  \end{array}
  \bigg).
  \end{equation*}
Thus, the unique invariant probability measure of the
continuous-time Markov chain $(\bar\alpha_t)$ generated by $\bar Q$
above is
\begin{equation}\label{v5}
\bar\pi=\Big(\frac{q-1}{p+q},\frac{p+1}{p+q}\Big).
\end{equation}
By \eqref{s4}, one has $\check\La=\La_1$ and $\hat\mu=\mu_2,$ which,
together with \eqref{a2},  leads to  $\Upsilon_i=\beta_i (
\La_1/\mu_2)^i$. On the other hand, \eqref{s5}  implies that
$i\mapsto\Gamma_i=\Upsilon_i-(\mu_i+\nu_i+\delta_i)$ is
nondecreasing. Furthermore, combining \eqref{s8} with \eqref{v4}
ensures $\Theta_5$, defined in \eqref{c11}, is   less than $1$.
Thus, the assertions \eqref{c14} and \eqref{g1} in
 Theorem \ref{state} hold true.   }
\end{exam}

\begin{exam}\label{exam4}
{\rm We continue to investigate  the model \eqref{1.1} with $G$
introduced in \eqref{oo}, where $(\alpha_t)_{t\ge0}$ is a
continuous-time jump process with the state space $\S=\{1,2\}$ and
the $Q$-matrix $Q(x)$ set by
\begin{equation*}
Q(x)= \bigg(\begin{array}{ccc}
 \frac{x^2}{1+x^2}-p & p-\frac{x^2}{1+x^2}\\
  q+\frac{x^2}{1+x^2} & -q-\frac{x^2}{1+x^2}\\
  \end{array}
  \bigg).
  \end{equation*}
In addition to  \eqref{s4}, we further suppose that
\begin{equation}\label{a0}
 \beta_1\La_1/\mu_2-(\mu_1+\nu_1+\delta_1)>0,~~~\beta_2 ( \La_1/\mu_2)^2-(\mu_2+\nu_2+\delta_2)<0
\end{equation}
and that
\begin{equation}\label{a7}
\frac{q+1}{p-1}< \frac{\mu_2+\nu_2+\delta_2-\beta_2 (
\La_1/\mu_2)^2}{\beta_1\La_1/\mu_2-(\mu_1+\nu_1+\delta_1)}.
\end{equation}

  Observe that
\begin{equation*}
Q^*=(q_{ij}^*)_{1\le i,j\le 2}= \left(\begin{array}{ccc}
 p-1 & p-1\\
  q+1 & -q-1\\
  \end{array}
  \right).
  \end{equation*}
Then, the continuous-time Markov chain $(\alpha^*_t)$ generated by
the $Q$-matrix $Q^*$ above possesses a unique invariant probability
measure
\begin{equation}\label{c13}
\pi^*=\Big(\frac{q+1}{p+q},\frac{p-1}{p+q}\Big).
\end{equation}
For the present setup, observe that $\Upsilon_i=\beta_i (
\La_1/\mu_2)^i$. From \eqref{a0}, we deduce that
$i\mapsto\Gamma_i=\Upsilon_i-(\mu_i+\nu_i+\delta_i)$ is
nonincreasing.  Moreover, \eqref{a7} and \eqref{c13} guarantee that
$\Theta_6$, introduced in \eqref{c12}, is smaller than $1.$ Hence,
  we can make a conclusion that the assertions \eqref{c14} and
  \eqref{g2} hold
 by virtue
of Theorem \ref{state}.
 }
\end{exam}

For another application of Lemma \ref{lm}, the following theorem
provides a criterion to determine the persistence  of the infectious
individuals modelled by \eqref{1.1} and \eqref{jump}.

\begin{thm}\label{th}
Assume the assumptions of Theorem \ref{t2.3} and Lemma \ref{lm}
hold.
 Suppose further that $i\mapsto\Gamma_i:=
\Lambda_i-(\mu_i+\nu_i+\delta_i)$ is nondecreasing and that
\begin{equation}\label{s2}
\Theta_7:=\frac{\sum_{i\in\S}\pi_i^*\Lambda_i}{\tau\sum_{i\in
\S}\pi_i^*(\mu_i+\nu_i+\delta_i)} >1,
\end{equation}
where $\tau>0$ is introduced in \eqref{d2}; or that
$i\mapsto\Theta_i$ is nonincreasing  and that
\begin{equation}\label{s3}
\Theta_8:=\frac{\sum_{i\in\S}\bar\pi_i\Lambda_i}{\tau\sum_{i\in
\S}\bar\pi_i(\mu_i+\nu_i+\delta_i)} >1.
\end{equation}
Then
\begin{equation}\label{s6}
\liminf_{t\rightarrow\infty}\Big(\frac{1}{t}\int_0^t I_s\d
s\Big)>0,\quad a.s.
\end{equation}
\end{thm}

\begin{proof}
We remark that \eqref{d3} still holds for the present framework.
Then, applying Lemma \ref{lm} and strong ergodic theorem for the
continuous-time Markov chains and taking the nondecreasing (resp.
nonincreasing) property of $i\mapsto\Gamma_i$ and \eqref{s2} (resp.
\eqref{s3}) into consideration yields the desired assertion
\eqref{s6}.
\end{proof}

\section{Extension to stochastic SIRS}\label{sec3}
In this section, we move forward to extend the random SIRS model
\eqref{1.1} and \eqref{v4} (or \eqref{jump}) with a specific $G$
into the stochastic SIRS model determined by \eqref{se3} and
\eqref{jump}.
 Throughout this section, we still let $N_t=S_t+I_t+R_t$ and
  $\Psi_i =\mu_i+\nu_i+\delta_i+(\mu_i^e)^2/2,i\in\S$.

 Our main result in this section is stated as follows, which
 provides some sufficient conditions to examine the extinction of
 the infectious individuals.

\begin{thm}\label{sto}
Assume the assumption $({\bf A3})$ and the conditions of Lemma
\ref{lm} hold. If $i\mapsto\Gamma_i:=\frac{\check\beta^e}{\hat\mu
}\Lambda_i-\Psi_i $ is nondecreasing and that
\begin{equation}\label{l1}
\Theta_9:=\frac{\sum_{i\in\S}\bar\pi_i\Gamma_i}{\sum_{i\in
\S}\bar\pi_i(\mu_i+\nu_i+\delta_i)} <1;
\end{equation}
or that $i\mapsto\Gamma_i  $ is nonincreasing and that
\begin{equation}\label{l2}
\Theta_{10}:=\frac{\sum_{i\in\S}\pi_i^*\Gamma_i}{\sum_{i\in
\S}\pi_i^*(\mu_i+\nu_i+\delta_i)} <1.
\end{equation}
Then
\begin{equation}\label{t3}
\lim_{t\rightarrow\infty}I_t=0~~~\mbox{ a.s.
}~~~\lim_{t\rightarrow\infty}R_t=0
\end{equation}
Moreover, if $i\mapsto\Lambda_i$ is nondecreasing, then
\begin{equation}\label{t1}
\sum_{i\in\S}\pi_i^*\Lambda_i\le
\liminf_{t\ra\infty}\Big(\int_0^t\mu_{\al_s}S_s\d
 s\Big)\le\limsup_{t\ra\infty}\Big(\int_0^t\mu_{\al_s}S_s\d
 s\Big)\le\sum_{i\in\S}\bar\pi_i\Lambda_i;
\end{equation}
or if $i\mapsto\Lambda_i$ is nondecreasing, then
\begin{equation}\label{t2}
\sum_{i\in\S}\bar\pi_i\Lambda_i\le
\liminf_{t\ra\infty}\Big(\int_0^t\mu_{\al_s}S_s\d
 s\Big)\le \limsup_{t\ra\infty}\Big(\int_0^t\mu_{\al_s}S_s\d
 s\Big)\le\sum_{i\in\S}\pi_i^*\Lambda_i.
\end{equation}
\end{thm}

Before we proceed to complete the proof of Theorem \ref{sto},  we
prepare some auxiliary lemmas.

\begin{lem}\label{lem1}
Assume that $({\bf A3})$   holds. Then, \eqref{se3} and \eqref{jump}
  has a unique strong solution
$(S_t,I_t,R_t)\in\R_+^3$ for the initial value
$(s_0,i_0,r_0)\in\R^3_+$. Moreover, for any
$p\in(1,1+2\hat\mu/(\check\mu^e)^2)$ there exists a constant $C_p>0$
such that
\begin{equation}\label{f1}
\sup_{k\in\mathbb{N}}\E\Big(\sup_{k\le s\le k+1}N^p_s\Big) \le C_p.
\end{equation}
\end{lem}

\begin{proof}
By following the argument of Lemma \ref{t2.1} and making use of the
Lyapunov function for  any $x=(x_1,x_2,x_3)\in\R^3_+$,
\begin{equation*}
V(x)=(x_1+x_2+x_3)^2+x_1-a-a\ln (x_1/a)+x_2-1-\ln x_2+x_3-1-\ln x_3
\end{equation*}
for some constant $a>0$ chosen suitably, we conclude that
\eqref{se3} and \eqref{jump} has a unique strong solution
$(S_t,I_t,R_t)\in\R^3_+$ for the initial value
$(s_0,i_0,r_0)\in\R^3_+$.

 From \eqref{se3}, it is obvious to see that
\begin{equation}\label{w9}
\d N_t=\{\La_{\al_t}-\mu_{\al_t}N_t-\delta_{\al_t}I_t\}\d
t-\mu_{\al_t}^eN_t\d B_t^{(1)},~~~t\ge0.
\end{equation}
In what follows, we fix $p\in(1,(1+2\hat\mu/(\check\mu^e)^2)]$. By
It\^o's formula, it follows that
\begin{equation}\label{w1}
\d
N_t^p=pN_t^{p-1}\{\La_{\al_t}-(\mu_{\al_t}-(p-1)(\mu_{\al_t}^e)^2/2)N_t-\delta_{\al_t}I_t\}\d
t-p\mu_{\al_t}^eN_t^p\d B_t^{(1)}.
\end{equation}
Taking advantage of $I_t\ge0$ and  Young's inequality: $a^\alpha
b^{1-\alpha}\le \alpha a+(1-\alpha)b, a,b\ge0, \alpha\in(0,1),$
yields that
\begin{equation*}
\begin{split}
\E N_t^p&\le  \E N_s^p+\int_s^t\{p\check\Lambda\E
N_u^{p-1}-(\hat\mu-(p-1)(\check\mu^e)^2/2))\E N_u^p\}\d u\\
&\le \E N_s^p+\int_s^t\{c_1-c_2\E N_u^p\}\d u,~~~~~t\ge
s\ge0,~p\in(1,1+2\hat\mu/(\check\mu^e)^2)
\end{split}
\end{equation*}
for some constants $c_1,c_2>0.$ Subsequently,  Gronwall's inequality
gives that for some $c_3>0,$
\begin{equation}\label{r2}
\sup_{t\ge0}\E N_t^p\le c_3.
\end{equation}
 Moreover, by BDG's inequality and Young's
inequality, we deduce from \eqref{w1}  that there exist constants
$c_4,c_5>0$ such that for any $t\in[k,k+1],$
\begin{equation*}
\begin{split}
\E\Big(\sup_{k\le s\le t}N^p_s\Big)&\le\E N^p_k+c_4\int_k^t\E
N^p_s\d
s+c_4\,\E\Big(\int_k^tN^{2p}_s\d s\Big)^{1/2}\\
&\le\E N^p_k+c_4\int_k^t\E N^p_s\d s+c_4\,\E\Big(\sup_{k\le
s\le t}N^p_s\int_k^tN^p_s\d s\Big)^{1/2}\\
&\le\frac{1}{2}\E\Big(\sup_{k\le s\le t}N^p_s\Big)+\E
N^p_k+c_5\int_k^t\E N^p_s\d s.
\end{split}
\end{equation*}
This, combining with \eqref{r2}, yields \eqref{f1}.
\end{proof}

\begin{lem}\label{lem}
Under the assumption $({\bf A3})$,
\begin{equation}\label{w5}
\lim_{t\rightarrow\infty}\Big(\frac{1}{t}\int_0^t\frac{\beta_{\al_s}^eS_s}{f(I_s)}\d
B_s^{(2)}\Big)=0~~~~\mbox{ a.s.}
\end{equation}
and, for $M_t:=\mu_{\al_s}N_t+\delta_{\al_s}I_t-\Lambda_{\al_t}$,
\begin{equation}\label{w6}
\lim_{t\rightarrow\infty}\Big(\frac{1}{t}\int_0^tM_s\d
s\Big)=0,~~~~~\mbox{ a.s.}
\end{equation}
\end{lem}

\begin{proof}
To derive \eqref{w5}, it suffices to verify  that
\begin{equation}\label{w2}
\lim_{k\rightarrow\infty}\Big(\frac{1}{k}\sup_{t\in[k,k+1]}\Big|\int_k^t\frac{\beta_{\al_s}^eS_s}{f(I_s)}\d
B_s^{(2)}\Big|\Big)=0~\mbox{ a.s.}
\end{equation}
and
\begin{equation}\label{w8}
\lim_{k\rightarrow\infty}\Big(\frac{1}{k}\int_0^k\frac{\beta_{\al_s}^eS_s}{f(I_s)}\d
B_s^{(2)}\Big)=0~~~~\mbox{ a.s. }
\end{equation}
Hereinafter, we stipulate
$p\in(1,(1+2\hat\mu/(\check\mu^e)^2)\wedge2)$. By BDG's inequality,
we deduce from \eqref{f1} and ({\bf A3}) that there exist  constants
$\tilde C_p,\hat C_p>0$ such that
\begin{equation}\label{w4}
\begin{split}
\E\Big(\sup_{t\in[k,k+1]}\Big|\int_k^t\frac{\beta_{\al_s}^eS_s}{f(I_s)}\d
B_s^{(2)}\Big|^p\Big)&\le\frac{\tilde
C_p(\check\beta^e)^p}{f^p(0)}\E\Big(\int_k^{k+1}S_s^2\d
s\Big)^{p/2}\\
&\le\frac{\tilde C_p(\check\beta^e)^p}{f^p(0)}\E\Big(\sup_{k\le s\le
k+1}S_s^p\Big) \le\hat C_p.
\end{split}
\end{equation}
For any $M>0$ and each integer $k\ge1$, set
\begin{equation*}
A_{k,M}:=\Big\{\frac{1}{k}\sup_{t\in[k,k+1]}\Big|\int_k^t\frac{\beta_{\al_s}^eS_s}{f(I_s)}\d
B_s^{(2)}\Big|\ge M\Big\}.
\end{equation*}
Via Chebyshev's inequality, it follows from  \eqref{w4} that
\begin{equation*}
\begin{split}
\mathbb{P}(A_{k,M}
)&\le\frac{1}{k^pM^p}\E\Big(\sup_{t\in[k,k+1]}\Big|\int_k^t\frac{\beta_{\al_s}^eS_s}{f(I_s)}\d
B_s^{(2)}\Big|^p\Big)\le\frac{\hat C_p}{k^pM^p},
\end{split}
\end{equation*}
Since the series $\sum_{k=1}^\infty1/k^p$ converges, by the
Borel--Cantelli lemma, we can conclude that
\begin{equation*}
\mathbb{P}\Big(\limsup_{k\rightarrow\infty}A_{k,M}\Big)=0.
\end{equation*}
Therefore, \eqref{w2} follows due to the arbitrariness of $M$.

Now we turn to claim that \eqref{w8} holds.  Let
\begin{equation*}
s_0=0,~~~~s_k=\int_0^k\frac{\beta_{\al_s}^eS_s}{f(I_s)}\d
B_s^{(2)},~~~x_k=s_k-s_{k-1},~~~~k\ge1.
\end{equation*}
Clearly, $\E(x_k/k)<\infty.$ On the other hand,  \eqref{w4} implies
that
\begin{equation}
\begin{split}
\sum_{k=1}^\infty\E(|x_k|^p|\F_{k-1})k^{-p}\le \hat
C_p\sum_{k=1}^\infty  k^{-p}<\infty.
\end{split}
\end{equation}
Thus, \eqref{w8} is available from \cite[Theorem 5]{chow}.

Following the arguments to derive \eqref{w2} and \eqref{w8},
respectively, we  deduce that
\begin{equation}\label{u}
\lim_{t\rightarrow\infty}\frac{N_t}{t}=0,~~~~~\mbox{
a.s.}~~~~~~\lim_{t\rightarrow\infty}\frac{1}{t}\Big(\int_0^t\mu_{\al_s}^eN_s\d
B_s^{(1)}\Big)=0,~~~~~~\mbox{ a.s.}
\end{equation}
Those, together with \eqref{w9},  yields \eqref{w6}.
\end{proof}

With Lemmas \ref{lem} and \ref{lem1} in hand, we are now in position
to finish the

\noindent{\bf Proof of Theorem \ref{sto}} For notation simplicity,
let
\begin{equation*}
I_1(t)=\frac{\check\beta^e}{\hat\mu f(0)}\cdot\frac{1}{t}\int_0^t
M_t \d t,~~~I_2(t)=- \int_0^t\mu_{\al_s}^e \d
B_s^{(1)},~~~I_3(t)=\frac{1}{t}\int_0^t\frac{\beta_{\al_s}^eS_s}{f(I_s)}\d
B_s^{(2)}.
\end{equation*}
By the It\^o formula, we find from ({\bf A3}) and $0\le S_t\le N_t$
that
\begin{equation*}
\begin{split}
\frac{1}{t}\ln
(I_t/I_0)&=\frac{1}{t}\int_0^t\Big\{\frac{\beta_{\al_s}^eS_s}{f(I_s)}-\frac{1}{2}\frac{(\beta_{\al_s}^e)^2S_s^2}{f(I_s)^2}-\Psi_{\al_s}
\Big\}\d s+\frac{I_2(t)}{t} +I_3(t)\\
&\le\frac{1}{t}\int_0^t\{\beta_{\al_s}^eS_s/f(0)-\Psi_{\al_s} \}\d
s+\frac{I_2(t)}{t} +I_3(t)\\
&\le I_1(t)+\frac{1}{t}\int_0^t\Big\{\frac{\check\beta^e}{\hat\mu
f(0)}\Lambda_{\al_s}-\Psi_{\al_s} \Big\}\d s+\frac{I_2(t)}{t}
+I_3(t).
\end{split}
\end{equation*}
By \eqref{w5} and \eqref{w6}, one has
\begin{equation}\label{l3}
\lim_{t\rightarrow\infty}(I_1(t)+I_3(t))=0,~~~~~\mbox{ a.s.}
\end{equation}
Let $\langle I_2\rangle_t$ be the quadratic variation of $I_2(t)$. A
simple calculation shows that
\begin{equation*}
\limsup_{t\rightarrow\infty}\frac{1}{t}\langle
I_2\rangle_t\le(\check\mu^e)^2
\end{equation*}
so that the strong law of large numbers for continuous martingales
gives that
\begin{equation}\label{l4}
\lim_{t\rightarrow\infty}\frac{1}{t}I_2(t)=0,~~~~~\mbox{ a.s.}
\end{equation}
Combining \eqref{l3} with \eqref{l4} and employing
$i\mapsto\Gamma_i$ is nondecreasing (resp. nonincreasing) and
\eqref{l1} (resp. \eqref{l2}), we deduce from Lemma \ref{lm} and the
strong ergodic theorem for continuous-time Markov chains that
\begin{equation*}
\lim_{t\rightarrow\infty}I_t=0~~~\mbox{ a.s. }
\end{equation*}
which implies that for any $\varepsilon\in(0,1)$ sufficiently small
there exist  $\Omega_1\subseteq\Omega$ with $\mathbb{P}(\Omega_1)=1$
and $T_1=T_1(\omega)>0$ such that
\begin{equation}\label{y1}
I_t(\omega)\le \varepsilon,~~~~~t\ge T_1,~~~~\omega\in\Omega_1.
\end{equation}
Set $ \xi_i:=\mu_i+\gamma_i+(\mu_i^e)^2/2,~i\in\S, $ and for any
$0\le s\le t$,
\begin{equation*}
\Lambda_{s,t}:=\int_s^t\mu_{\al_u}^e\d B_u^{(1)} ~~\mbox{ and
}~~\Phi_{s,t}:=\exp\Big(-\int_s^t\xi_{\al_u} \d
u-\Lambda_{s,t}\Big).
\end{equation*}
By \eqref{l4}, we have
\begin{equation*}
\lim_{t\rightarrow\infty}\frac{1}{t}\ln\Phi_{0,t}<0,~~~~\mbox{
a.s.},
\end{equation*}
which implies that there exists $\Omega_2\subseteq\Omega$ with
$\mathbb{P}(\Omega_2)=1$ and $T_2=T_2(\omega)>0$  such that
\begin{equation}\label{t4}
\Phi_{0,t}(\omega)\le \varepsilon,~~~~~t\ge
T_2,~~~~\omega\in\Omega_2.
\end{equation}
Next, using the law of the iterated logarithm for stochastic
integrals \cite[(1.2)]{Wang}, we deduce that
\begin{equation*}
\liminf_{t\rightarrow\infty}\frac{\Lambda_{s,t}}{\sqrt{2\langle
\Lambda\rangle_{s,t}\ln\ln\langle
\Lambda\rangle_{s,t}}}=-1~~~~\mbox{ a.s. }
\end{equation*}
So there exists $\Omega_3\subseteq\Omega$ with
$\mathbb{P}(\Omega_3)=1$, $T_3=T_3(\omega)>0$  such that
\begin{equation}\label{y2}
-(1+\varepsilon)\sqrt{2\langle \Lambda\rangle_{s,t}\ln\ln\langle
\Lambda\rangle_{s,t}}\le\Lambda_{s,t}\le
(-1+\varepsilon)\sqrt{2\langle \Lambda\rangle_{s,t}\ln\ln\langle
\Lambda\rangle_{s,t}},~~t,s\ge T_2,~~\varepsilon\in\Omega_3.
\end{equation}
By the variation-of-constants formula (see e.g. \cite[Theorem
3.1]{M08}), we deduce from \eqref{y1} and \eqref{y2} that for any
$\omega\in\Omega_0:=\Omega_1\cap\Omega_2\cap\Omega_3$ and $t\ge
T:=T_1+T_2+T_3$,
\begin{equation}\label{y3}
\begin{split}
R_t(\omega)&=\Phi_{0,t}(\omega) \Big\{R_0+
\int_0^t\nu_{\al_s}(\omega)
  I_s(\omega)\Phi_{0,s}^{-1}(\omega)\d s\Big\} \\
&\le R_0\,\Phi_{0,t}(\omega) +\check\nu \,\Phi_{T,t}(\omega)
\int_0^T
  I_s(\omega)\Phi_{s,T}(\omega)\d s +\check\nu\int_T^t
  I_s(\omega)\Phi_{s,t}(\omega)\d s\\
  &\le R_0\,\varepsilon+\check\nu\,\varepsilon\int_T^t
   \exp\Big(-\int_s^t\xi_{\al_u} \d
u+(1+\varepsilon)\sqrt{2\langle \Lambda\rangle_{s,t}\ln\ln\langle
\Lambda\rangle_{s,t}}\Big)\d s\\
&\quad+\check\nu \int_0^T
  I_s(\omega)\Phi_{s,T}(\omega)\d s \exp\Big(-\int_T^t\xi_{\al_u} \d
u+(1+\varepsilon)\sqrt{2\langle \Lambda\rangle_{T,t}\ln\ln\langle
\Lambda\rangle_{T,t}}\Big).
\end{split}
\end{equation}
Furthermore, observe that there exist constants $c>0$ and
$\alpha\in(0,1)$ such that
\begin{equation}\label{y5}
(1+\varepsilon)\sqrt{2\langle \Lambda\rangle_{s,t}\ln\ln\langle
\Lambda\rangle_{s,t}}\le c+\alpha\int_s^t\xi_{\al_u} \d u.
\end{equation}
Plugging \eqref{y5} into \eqref{y3} and making use of the
arbitrariness of $\varepsilon\in(0,1)$, we obtain   \eqref{t3}.

With \eqref{t3} in hand, we deduce from \eqref{w9} and \eqref{u}
that
\begin{equation*}
\lim_{t\rightarrow\infty}\Big(\frac{1}{t}\int_0^t(\Lambda_{\al_s}-\mu_{\al_s}S_s)\d
s\Big)=0
\end{equation*}
This, combining with Lemma \ref{lm}, yields the assertions
\eqref{t1} and \eqref{t2}.

\noindent{\bf Acknowledgement.} This work is supported by NNSFs of China (Nos.  11771327, 11301030, 11431014, 11831014).

\end{document}